\documentclass{article}    % Specifies the document style.
\usepackage{graphics}
\usepackage{graphicx}
\usepackage{enumerate}
\usepackage{multirow}
\usepackage{epsfig}
\usepackage{anysize}
\usepackage{array}
\usepackage{authblk}
\usepackage{latexsym,amssymb,amsmath,amscd,amsthm,amsxtra}
\usepackage{subfigure}
\usepackage{tabularx}
\usepackage[table]{xcolor}
\definecolor{lightgray}{gray}{0.9}
\newtheorem{theorem}{Theorem}
\newtheorem{problem}{Problem}
\newtheorem{lemma}{Lemma}[section]

\numberwithin{equation}{section}

\def\({\left( }
\def\){\right )}
\graphicspath{ {figures/} }
\begin{document}
\title {\textbf{Non-Polynomial Quintic Spline for Numerical Solution of Fourth--Order Time Fractional Partial Differential Equations}}
%\author[1]{Muhammad Amin}%
\author{Muhammad Abbas\footnote{Corresponding author. Mobile: +92 304 6282830, e-mail addresses: m.abbas@uos.edu.pk.}}
%\author[3]{Muhammad Kashif Iqbal}
%\affil[1]{\small Department of Mathematics, National College of Business Administration \& Economics, Lahore, Pakistan.}
\affil[2]{\small Department of Mathematics, University of Sargodha, Sargodha, Pakistan.}
%\affil[3]{\small Department of Mathematics, Government College University, Faisalabad, Pakistan.}
%\date{today}
\maketitle
\vspace{-0.5cm}
\begin{abstract}
This paper presents a novel approach for numerical solution of a class of fourth order time fractional partial differential equations (PDE's). The finite difference formulation has been used for temporal discretization, whereas, the space discretization is achieved by means of non polynomial quintic spline method. The proposed algorithm is proved to be stable and convergent. In order to corroborate this work, some test problems have been considered and the computational outcomes are compared with those found in the exiting literature. It is revealed that the presented scheme is more accurate as compared to current variants on the topic.
\end{abstract}

\maketitle

%%%%%%%%%%%%%%%%%%%%%%%%%%%%%%%%%%%%%%%%%%%%%%%%%%%%%%%%%%%%%%%%%

% Type in your PAPER, starting below:

\section{Introduction}
In the modern era, fractional order differential equations have gained a significant amount of research work due to their vide range of applications in various branches of science and engineering such as Physics, electrical networks, fluid mechanics, control theory, theory of viscoelasticity, neurology and theory of electromagnetic acoustics \cite{podlubny1998,miller1993}. Wang  \cite{wang2006} introduced the very first approximate solution of nonlinear fractional Korteweg--de Vries (KdV) Burger equation involving space and time fractional derivatives using Adomian Decomposition method. Zurigat \emph{at al.} \cite{zurigat2010} examined the approximate solution of fractional order algebraic differential equations using Homotopy analysis method. Turut and Guzel \cite{turut2013} implemented Adomian decomposition method and multivariate Pade approximation method for solving fractional order nonlinear partial differential equations (PDE's). In \cite{liu2011}, Liu and Hou applied the Generalized differential transform method to solve the coupled Burger equation with space and time fractional derivatives. Khan \emph{et al.} \cite{khan2011} used Adomian decomposition method and Variational iteration method for numerical solution of fourth order time fractional PDE's with variable coefficients. Later on, Abbass \emph{et al.} \cite{abbas2014} employed a finite difference approach based on third degree trigonometric B-spline functions for approximate solution of one-dimensional wave equation.
Javidi and Ahmad \cite{javidi2015} developed a computational technique based on Homotopy perturbation method Laplace transform and Stehfest’s numerical inversion algorithm for solving fourth-order time-fractional PDE's with variable coefficients.The fractional differential transform method and modified fractional differential transform method were proposed by Kanth and Aruna \cite{kanth2015} for series solution to higher dimensional third-order dispersive fractional PDE's. Pandey and Mishra \cite{pandey2017} applied Sumudu transforms and Homotopy analysis approach for solving time-fractional third order dispersive type of PDE's. The fractional Variational iteration method was put into action by Prakash and Kumar in \cite{prakash2017} for series solution to third-order fractional dispersive PDE's in higher dimensional space.\\
The spline approximation techniques have been applied extensively for numerical solution of ODE's and PDE's. The spline functions have a variety of significant gains over finite difference schemes. These functions provide a continuous differentiable estimation to solution over the whole spatial domain with great accuracy. The straightforward employment of spline functions provides a solid ground for applying them in the context of numerical approximations for initial/boundary problems.\\
Khan and Aziz \cite{khan2003}, solved third order boundary-value problems (BVP's) using a numerical method based on quintic spline functions.
In \cite{ramadan2009}, non polynomial quintic spline method was employed for numerical solution of fourth order two-point BVP's. Khan and Sultana \cite{khan2012} proposed non-polynomial quintic spline functions for numerical solution of third order BVP's associated with odd-order obstacle problems. In \cite{srivastava2014}, Srivastava discussed numerical solution of differential equations  using polynomial spline functions of different orders. Siddiqi and Arshed \cite{siddiqi2015} brought the fifth degree basis spline collocation functions into use for approximate solution of fourth order time fractional PDE's. Rashidinia and Mohsenyzadeh \cite{rashidinia2015} used non-polynomial quintic spline technique for one-dimensional heat and wave equations. Recently, in \cite{tariq2017}, fifth degree spline approximation technique has been utilized for approximate solution of fourth-Order time-fractional PDE's. In \cite{hamasalh2015}, the new fractional order spline functions were considered to obtain the approximate solution for fractional Bagely-Torvik Equation. Arshed \cite{arshed2017} employed quintic B-spline collocation scheme for solving fourth order time-fractional super diffusion equation. More recently, parametric quintic spline approach and Grunwald-Letnikov approximation have been proposed in \cite{li2018} for a distributed order fractional sub-diffusion problem.\\
In the field of modern science and engineering the fourth-order initial/boundary value problems are of great importance. For example, airplane wings, bridge slabs, floor systems and window glasses are being  modeled as plates subject to different types end supports which are successfully described in terms of fourth-order PDE's \cite{tariq2017}. In this work, we consider the following class of the fourth-order time-fractional PDE's
\begin{equation}\label{1e}
  \frac{\partial^{\gamma} y}{\partial t^{\gamma}}+\alpha \frac{\partial^{4} y}{\partial x^{4}}=u(x,t), \ \ \ t\in[0,T],\ \ \ \ x\in[0,L],
\end{equation}
with the following initial and boundary conditions
\begin{equation*}
y(x,0)=v_{0}(x)
\end{equation*}
\begin{equation*}
  y(0,t)=y(L,t)=0
\end{equation*}
\begin{equation*}
  y_{xx}(0,t)=y_{xx}(L,t)=0
\end{equation*}
where $\gamma \in(0,1)$, is the order of fractional time derivative, $\alpha$ represents the ratio of flexural-rigidity of beam to its mass per unit length, $y(x,t)$ is the beam transverse displacement, $u(x,t)$ describes the dynamic driving force per unit mass and the function $v_{0}(x)$ is known to be continuous on $[0,L]$. There are many descriptions to the concept of fractional differentiation but Caputo and Riemann-Liouville have been the most common definitions. Here, we shall use the Caputo's approach because it is more appropriate for real world problems and it permits initial and boundary conditions in terms of ordinary derivatives. The Caputo's definition of fractional derivative of order $\gamma$ is given by
\begin{equation*}
 \frac{\partial ^{\gamma}y(x,t)}{\partial t^{\gamma}}=
  \begin{cases}
    \frac{1}{\Gamma(1-\gamma)}\int\limits_{0}^{t}\frac{\partial y(x,s)}{\partial s}\frac{ds}{(t-s)^\gamma} & \mbox{, } 0<\gamma<1 \\
    \frac{\partial y(x,t)}{\partial t},  \ \ \ \ \ \ \ & \mbox{$\gamma=1$}.
  \end{cases}
 \end{equation*}
This paper has been composed with the aim to develop a spline collocation method for approximate solution of fourth order time-fractional PDE's. The backward Euler's scheme has been utilized for temporal discretization, whereas, non polynomial quintic spline function, comprised of a trigonometric part and a polynomial part, has been used to interpolate the unknown function in spatial direction. The presented technique has also been proved to be stable and convergent. \\
This work is arranged as follows: In section 2, a brief explanation of quintic spline scheme has been presented and the consistency relations between the values of spline approximation and its derivatives at the nodal points are derived. Section 3 describes the use of $L1$ approximation in time direction to achieve a backward Euler technique. Non-polynomial quintic spline scheme for the spatial discretization has been discussed in section 4. The computational results and discussions are given in section 5.
\section{Description of Non Polynomial Quintic Spline Function}
Consider $x_{i}=ih$, be the mesh points of uniform partition of $[0,L]$ into sub-intervals $[x_{i},x_{i-1}]$, where $h=\frac{L}{n}$ and $i=0,1,2,\cdots,n$. Let $y(x)$ be a sufficiently smooth function defined on $[0,L]$. We denote the non polynomial quintic spline approximation to $y(x)$ by $S(x)$. Each non polynomial spline segment $R_{i}(x)$ has the following form
 \begin{equation}\label{21e}
 R_{i}(x)=a_{i}\cos (\xi(x-x_{i}))+b_{i}\sin (\xi(x-x_{i}))+c_{i}(x-x_{i})^{3}+d_{i}(x-x_{i})^{2}+e_{i}(x-x_{i})+f_{i},   \\
                                                                                                i=0,1,2,\cdots,n.
 \end{equation}
where $a_{i},b_{i},c_{i},d_{i},e_{i}$ and $f_{i}$ are the constants and the parameter $\xi$, the frequency of the trigonometric functions, will be used to enhance the accuracy of the technique. When $\xi$ approaches to zero, Eq.\eqref{21e} reduces to quintic polynomial spline function in $[a,b]$. The non polynomial quintic spline can be defined as

  \begin{equation}\label{22e}
  S(x)=R_{i}(x),~~~~  \forall~ x\in[x_{i},x_{i+1}],    \ i=0,1,2,\cdots,n.
  \end{equation}
  \begin{equation}
    R_{i}(x)\in C^{4}[0,L]
  \end{equation}
 First of all, we establish the consistency relations for all the coefficients involved in \eqref{21e} in terms of $S_{i}$'s, $M_{i}$'s and $F_{i}$'s, where
  \begin{equation*}
  S_{i}=S(x_{i})=R_{i}(x_{i}),
  \end{equation*}
  \begin{equation*}
  M_{i}=S^{''}(x_{i})=R^{''}_{i}(x_{i})
  \end{equation*}
  \begin{equation*}
  F_{i}=S^{(4)}(x_{i})=R^{(4)}_{i}(x_{i})
  \end{equation*}
 The values of coefficients introduced in \eqref{21e} can be calculated as
\begin{align*}
  a_{i} & =\frac{h^{4}}{\theta^{4}}F_{i}, \\
  b_{i} & =\frac{h^{4}}{\theta^{4}\sin(\theta)}( F_{i+1}-F_{i}\cos(\theta)), \\
  c_{i} & =\frac{1}{6h}(M_{i+1}-M_{i})+\frac{h}{6\theta^{2}}(F_{i+1}-F_{i}), \\
  d_{i} & =\frac{1}{2}M_{i}+\frac{h^{2}}{2\theta^{2}}F_{i}, \\
  e_{i} & =\frac{1}{h}(S_{i+1}-S_{i})+(\frac{h^{3}}{\theta^{4}}-\frac{h^3}{3\theta^{2}})F_{i}-(\frac{h^{3}}{\theta^{4}},
  +\frac{h^3}{6\theta^{2}})F_{i+1}-\frac{h}{6}(M_{i+1}+2M_{i}),\\
  f_{i} & =S_{i}-\frac{h^{4}}{\theta^{4}}F_{i},
\end{align*}
 where $\theta=\xi h$ and $i=0,1,\cdots,n-1.$\\
 Now, using the first and third derivative continuity conditions at the knots, i.e. $ R^{(\tau)}_{i-1}(x_{i})= R^{(\tau)}_{i}(x_{i})$, for $\tau=1,3$, we can derive the following important relations
  \begin{multline}\label{23e}
 M_{i-1}+4M_{i}+M_{i+1}=\frac{6}{h^{2}}(S_{i-1}-2S_{i}+S_{i+1})+\frac{6h^{2}}{\theta^{2}}(\frac{1}{\theta\sin(\theta)}-\frac{1}{\theta^{2}}-\frac{1}{6})(F_{i+1}+F_{i-1})\\
 +\frac{6h^{2}}{\theta^{2}}(\frac{2}{\theta^{2}}-\frac{2\cos(\theta)}{\theta\sin(\theta)}-\frac{4}{6})F_{i}
 \end{multline}
 and
  \begin{multline}\label{24e}
 M_{i-1}-2M_{i}+M_{i+1}=h^{2}(\frac{1}{\theta\sin(\theta)}-\frac{1}{\theta^{2}})(F_{i+1}+F_{i-1})+2h^{2}(\frac{1}{\theta{2}}-\frac{\cos(\theta)}{\theta\sin(\theta)})F_{i}
 \end{multline}
 Solving \eqref{23e} and \eqref{24e}, we get
  \begin{multline}\label{25e}
 M_{i}=\frac{1}{h^{2}}(S_{i-1}-2S_{i}+S_{i+1})+h^{2}(\frac{1}{\theta^{3}\sin(\theta)}-\frac{1}{6\theta\sin(\theta)}-\frac{1}{\theta^{4}})(F_{i+1}+F_{i-1})\\
 +h^{2}(\frac{2}{\theta^{4}}-\frac{2\cos(\theta)}{\theta^{3}\sin(\theta)}+\frac{2\cos(\theta)}{6\theta\sin(\theta)}-\frac{1}{\theta^{2}})(F_{i})
 \end{multline}
 Using \eqref{24e}--\eqref{25e}, we get the following consistency relation involving $F_{i}$ and $S_{i}$\\
  for $i=2,3,\cdots,n-2.$
 \begin{equation}\label{26e}
 S_{i+2}-4S_{i+1}+6S_{i}-4S_{i-1}+S_{i-2}=h^{4}(\alpha_{1} F_{i-2}+\beta_{1} F_{i-1}+\gamma_{1} F_{i}+\beta_{1} F_{i+1}+\alpha_{1} F_{i+2})
 \end{equation}
 where
 \begin{align*}
 \alpha_{1} & =(\frac{1}{\theta^{4}}+\frac{1}{6\theta\sin(\theta)}-\frac{1}{\theta^{3}\sin(\theta)}),\ \  \beta_{1}=(\frac{2+2\cos(\theta)}{\theta^{3}\sin(\theta)}+\frac{2-\cos(\theta)}{3\theta\sin(\theta)}-\frac{4}{\theta^{4}}) \\
 \gamma_{1} & =(\frac{1-4\cos(\theta)}{3\theta\sin(\theta)} - \frac{2+4\cos(\theta)}{\theta^{3}\sin(\theta)} + \frac{6}{\theta^{4}})
 \end{align*}
The relation \eqref{26e} provides $(n-3)$ linear equations with $(n-1)$ unknowns $S_{i}, i=1(1)n-1$. Hence, we require two more equations for direct calculation of $S_{i}$, one at each end of the range of integration, which can be formulated as\\
 setting $i=1,2$  in \eqref{23e} we have
 \begin{equation}\label{27e}
 M_{0}+4M_{1}+M_{2}=\frac{6}{h^{2}}(S_{0}-2S_{1}+S_{2})+\overset{\sim}{\lambda}(F_{0}+F_{2})+\overset{\sim}{\mu}F_{1}
 \end{equation}
 and\\
 \begin{equation}\label{28e}
    M_{1}+4M_{2}+M_{3}=\frac{6}{h^{2}}(S_{1}-2S_{2}+S_{3})+\overset{\sim}{\lambda}(F_{1}+F_{3})+\overset{\sim}{\mu}F_{2}
 \end{equation}
 Similarly, for $i=1,2$ the expression \eqref{24e} returns the following two equations
 \begin{equation}\label{29e}
   M_{0}-2M_{1}+M_{2}=\overset{\approx}{\lambda}(F_{0}+F_{2})+\overset{\approx}{\mu}F_{1}
 \end{equation}
 and
 \begin{equation}\label{30e}
   M_{1}-2M_{2}+M_{3}=\overset{\approx}{\lambda}(F_{1}+F_{3})+\overset{\approx}{\mu}F_{2}.
 \end{equation}
 where
 \begin{align*}
   \overset{\sim}{\lambda} & =\frac{6h^{2}}{\theta^{2}}(\frac{1}{\theta\sin\theta}-\frac{1}{\theta^{2}}-\frac{1}{6}),\ \ \    \overset{\sim}{\mu}=\frac{6h^{2}}{\theta^{2}}(\frac{2}{\theta^{2}}-\frac{2\cos(\theta)}{\theta\sin(\theta)}-\frac{4}{6}), \\
   \overset{\approx}{\lambda} & =h^{2}(\frac{1}{\theta\sin(\theta)}-\frac{1}{\theta^{2}})\ \ \ \text{and}\ \ \   \overset{\approx}{\mu}= 2h^{2}(\frac{1}{\theta{2}}-\frac{\cos(\theta)}{\theta\sin(\theta)}).
 \end{align*}
 From \eqref{27e} and \eqref{29e}, we have \\
\begin{equation}\label{31e}
  M_{1}=\frac{1}{h^{2}}(S_{0}-2S_{1}+S_{2})+\frac{\overset{\sim}{\lambda}-\overset{\approx}{\lambda}}{6}(F_{0}+F_{2})+\frac{\overset{\sim}{\mu}-\overset{\approx}{\mu}}{6}F_{1}
\end{equation}
Similarly, subtracting \eqref{30e} from \eqref{28e}, we get
\begin{equation}\label{32e}
   M_{2}= \frac{1}{h^{2}}(S_{1}-2S_{2}+S_{3})+\frac{\overset{\sim}{\lambda}-\overset{\approx}{\lambda}}{6}(F_{1}+F_{3})+\frac{\overset{\sim}{\mu}-\overset{\approx}{\mu}}{6}F_{2}
\end{equation}
Now, the first end condition is obtained by substituting \eqref{31e}, \eqref{32e} into \eqref{27e} for $i=1$.
\begin{equation}\label{33e}
  -2S_{0}+5S_{1}-4S_{2}+S_{3}=-h^{2}M_{0}+h^{4}(\omega_{0}F_{0}+\omega_{1}F_{1}+\omega_{2}F_{2}+\omega_{3}F_{3})
\end{equation}
Similarly, the second end condition for $i=n$, is given by
\begin{equation}\label{34e}
  S_{n-3}-4S_{n-2}+5S_{n-1}-2S_{n}=-h^{2}M_{n}+h^{4}(\omega_{3}F_{n-3}+\omega_{2}F_{n-2}+\omega_{1}F_{n-1}+\omega_{0}F_{n})
\end{equation}
where
\begin{align*}
  \omega_{0} & =(\frac{2}{\theta^{3}\sin(\theta)}-\frac{2}{\theta^{4}}+\frac{4}{6\theta\sin(\theta)}-\frac{1}{\theta^{2}}),\ \ \ \omega_{1}=\frac{1-8\cos(\theta)}{6\theta\sin(\theta)}-\frac{1+4\cos(\theta)}{\theta^{3}\sin(\theta)}+\frac{5}{\theta^{4}} \\
  \omega_{2} & = (\frac{2+2\cos(\theta)}{\theta^{3}\sin(\theta)}+\frac{2-\cos(\theta)}{3\theta\sin(\theta)}-\frac{4}{\theta^{4}}),\ \ \ \omega_{3}=\frac{1}{6\theta\sin(\theta)}-\frac{1}{\theta^{3}\sin(\theta)}+\frac{1}{\theta^{4}}
\end{align*}\\
\begin{lemma}\label{l1}
 The local truncation error $t_i, i=1(1)n-1$ associated with the Eqs \eqref{26e}, \eqref{33e} and \eqref{34e} is given by
 \begin{equation}\label{E317}
 t_{i}=\begin{cases}
  \big(\frac{11}{12}-\omega_0-\omega_1-\omega_2-\omega_3\big)h^4y_i^{(4)}+ \big(\frac{1}{12}+\omega_0-\omega_2-2\omega_3\big)h^5y_i^{(5)}&\\
  +\big(\frac{11}{90}-\frac{1}{2}\omega_0-\frac{1}{2}\omega_2-2\omega_3\big)h^6y_i^{(6)}
  +\big(\frac{1}{60}+\frac{1}{6}\omega_0-\frac{1}{6}\omega_2-\frac{4}{3}\omega_3\big)h^7y_i^{(7)}&\\
  +\big(\frac{17}{2240}-\frac{1}{24}\omega_0-\frac{1}{24}\omega_2-\frac{2}{3}\omega_3\big)h^8y_i^{(8)}+O(h^9),& i=1   \\

  \big(1-2\alpha_1-2\beta_1-\gamma_1\big)h^4y_i^{(4)}+\big(\frac{1}{6}-4\alpha_1-\beta_1\big)h^6y_i^{(6)}&\\
  +\big(\frac{1}{180}-\frac{4}{3}\alpha_1-\frac{1}{12}\beta_1\big)h^8y_i^{(8)}
  +\big(\frac{17}{30240}-\frac{8}{45}\alpha_1-\frac{1}{360}\beta_1\big)h^{10}y_i^{(10)}+O(h^{11}),& i=2(1)n-2  \\

  \big(\frac{11}{12}-\omega_0-\omega_1-\omega_2-\omega_3\big)h^4y_i^{(4)}+ \big(\frac{1}{12}+\omega_0-\omega_2-2\omega_3\big)h^5y_i^{(5)}&\\
  +\big(\frac{11}{90}-\frac{1}{2}\omega_0-\frac{1}{2}\omega_2-2\omega_3\big)h^6y_i^{(6)}
  +\big(\frac{1}{60}+\frac{1}{6}\omega_0-\frac{1}{6}\omega_2-\frac{4}{3}\omega_3\big)h^7y_i^{(7)}&\\
  +\big(\frac{17}{2240}-\frac{1}{24}\omega_0-\frac{1}{24}\omega_2-\frac{2}{3}\omega_3\big)h^8y_i^{(8)}+O(h^9),& i=n-1   \\
 \end{cases}
 \end{equation}
 \begin{proof}
 We have to find local truncation error $ t_i, i=1,2,...,n-1 $ for the present scheme. First of all, we write Eqs \eqref{26e}, \eqref{33e}, and \eqref{34e} as
\begin{align*}
  t_1&=-2y_0+5y_1-4y_2+y_3+h^2M_0-h^4\big(\omega_0 y_0^{(4)}+\omega_1y_1^{(4)}+\omega_2y_2^{(4)}+\omega_3 y_3^{(4)}\big),\\
  t_i&=y_{i-2}-4y_{i-1}+6y_i-4y_{i+1}+y_{i+2}-h^4\big(\alpha_1y_{i-2}^{(4)}+\beta_1y_{i-1}^{(4)}+\gamma_1 y_i^{(4)}+\beta_1 y_{i+1}^{(4)}+\alpha_1y_{i+2}^{(4)}\big), \\
  t_{n-1}&= y_{n-3}-4y_{n-2}+5y_{n-1}-2y_n+h^2M_n+h^4\big(\omega_3 y_{n-3}^{(4)}+\omega_2 y_{n-2}^{(4)}+\omega_1\big)y_{n-1}^{(4)}+\omega_0 y_{n}^{(4)}\big)
\end{align*}
The expressions for $ t_i, i=1,2,...,n-1 $ can be obtained by expanding the terms $y_0, y_1, y_1^{(4)}, y_2, y_2^{(4)}, y_3, y_3^{(4)}$ etc about the points $x_i,i=1,2,...,n-1$, using Taylor series respectively.
 \end{proof}
 \end{lemma}
 Equating the coefficients of $y_i^{(\tau)}$ for $\tau=4,5,6,7$, we get

$\alpha_1=-\frac{1}{720}, \beta_1=\frac{31}{180},\gamma_1=\frac{79}{120}, \omega_0=\frac{7}{90}, \omega_1=\frac{49}{72}, \omega_2=-\frac{7}{45}$ and $\omega_3=\frac{1}{360}$\\

The local truncation error given in Eq \eqref{E317} takes the following form
\begin{equation}\label{E319}
 t_{i}=\begin{cases}
  -\frac{241}{60480} h^8y_i^{(8)}+O(h^9),& i=1   \\
  \frac{1}{3024} h^{10}y_i^{(10)}+O(h^{11}),& i=2(1)n-2  \\
  -\frac{241}{60480} h^8y_i^{(8)}+O(h^9),& i=n-1   \\
 \end{cases}
 \end{equation}
\section{Temporal Discretization}
In order to discretize the time fractional derivative, backward Euler scheme is employed. We consider, $t_{p}=p\Delta t$ for $p=0(1)K$ with $\Delta t=\frac{T}{K}$ as the step size in time direction.
The computation of Caputo time-fractional derivative at $t=t_{p+1}$ can be made as
\begin{equation*}\label{3.1e}
  \int\limits_{0}^{t_{p+1}}\frac{\partial y(x,w)}{\partial w}(t_{p+1}-w)^{-\gamma}dw=\sum_{j=0}^{p}\int\limits_{t_{j}}^{t_{j+1}}\frac{\partial y(x,w)}{\partial w}(t_{p+1}-w)^{-\gamma}dw
\end{equation*}
\begin{align*}
  \int\limits_{0}^{t_{w+1}}\frac{\partial y(x,w)}{\partial w}(t_{p+1}-w)^{-\gamma}dw & =\sum_{j=0}^{p}\int\limits_{t_{j}}^{t_{j+1}}\frac{\partial y(x,w)}{\partial w}(t_{p+1}-w)^{-\gamma}dw \\
  = & \sum_{j=0}^{p}\frac{y(x,t_{j+1})-y(x,t_{j})}{\Delta t}\int\limits_{t_{j}}^{t_{j+1}}(t_{p+1}-w)^{-\gamma}dw+l_{\Delta t}^{p+1} \\
  = & \sum_{j=0}^{p}\frac{y(x,t_{j+1})-y(x,t_{j})}{\Delta t}\int\limits_{t_{p-j}}^{t_{p-j+1}}(\upsilon)^{-\gamma}d\upsilon+l_{\Delta t}^{p+1} \\
  = & \sum_{j=0}^{p}\frac{y(x,t_{p-j+1})-y(x,t_{p-j})}{\Delta t}\int\limits_{t_{j}}^{t_{j+1}}(\upsilon)^{-\gamma}d\upsilon+l_{\Delta t}^{p+1} \\
  = & \frac{1}{1-\gamma}\sum_{j=0}^{p}\frac{y(x,t_{p-j+1})-y(x,t_{p-j})}{\Delta t}((j+1)^{1-\gamma}-j^{1-\gamma})+l_{\Delta t}^{p+1} \\
  = & \frac{1}{1-\gamma}\sum_{j=0}^{p} b_{j}\frac{y(x,t_{p-j+1})-y(x,t_{p-j})}{\Delta t}+l_{\Delta t}^{p+1}
\end{align*}
Where
$b_{j}=(j+1)^{1-\gamma}-j^{1-\gamma}$ and $\upsilon=(t_{p+1}-w)$. The above equation along with the definition of Caputo fractional derivative gives the following relation.
\begin{equation}\label{3.3e}
  \frac{\partial^{\gamma} y(x,t_{p+1})}{\partial t^{\gamma}}=\frac{1}{\Gamma(2-\gamma)}\sum_{j=0}^{p} b_{j}\frac{y(x,t_{p-j+1})-y(x,t_{p-j})}{\Delta t^\gamma}+l_{\Delta t}^{p+1}
\end{equation}
Now, we define a semi--discrete fractional differential operator $G_{t}^{\gamma}$ as
\begin{equation*}
  G_{t}^{\gamma}y(x,t_{p+1})=\frac{1}{\Gamma(2-\gamma)}\sum_{j=0}^{p} b_{j}\frac{y(x,t_{p-j+1})-y(x,t_{p-j})}{\Delta t^\gamma}
\end{equation*}
Then, Eq. \eqref{3.3e} can be written as
\begin{equation}\label{38e}
   \frac{\partial^{\gamma} y(x,t_{p+1})}{\partial t^{\gamma}}=G_{t}^{\gamma}y(x,t_{p+1})+l_{\Delta t}^{p+1}
\end{equation}
Here, $l_{\Delta t}^{p+1}$ denotes the truncation error between $\frac{\partial^{\gamma}}{\partial t^{\gamma}}y(x,t_{p+1})$ and $G_{t}^{\gamma}y(x,t_{p+1})$. Let $G_{t}^{\gamma}y(x,t_{p+1})$ be the approximation of Caputo time-fractional derivative at $t=t_{p+1}$, then Eq. \eqref{1e} can be expressed as
\begin{equation}\label{39e}
  G_{t}^{\gamma}y(x,t_{p+1})+\alpha\frac{\partial^{4}}{\partial x^4
  }y(x,t_{p+1})=u(x,t_{p+1})
\end{equation}
Using \eqref{3.3e}, the above equation can be written as
\begin{multline}\label{3.4e}
   y^{p+1}(x)+\beta \alpha y^{p+1}_{xxxx}=(b_{0}-b_{1})y^{p}(x)+\sum_{j=1}^{p-1}(b_{j}-b_{j+1})y^{p-j}(x)+b_{p}y^{0}(x)+\beta u^{p+1}(x),\\  p=1,2,3,\cdots,j-1.
\end{multline}
where, $\beta=\Gamma(2-\gamma)\Delta t^{\gamma}$ and \ $y^{p+1}(x)=y(x,t^{p+1})$ with the initial and boundary conditions as follow
\begin{equation*}\label{3.5e}
  y^{0}=v_{0}(x),\ \ \ \  x\in [0,L].
\end{equation*}
Moreover, the coefficients $b_{j}$ involved in \eqref{3.3e} have the following properties \\
$\bullet \  b_{j}'s$ are non-negative for $ j=0,1,\cdots,p $ \\
$\bullet \ 1=b_{0}>b_{1}>b_{2}>b_{3}>\cdots>b_{p}, \  b_{p}\rightarrow0$ as $p \rightarrow\infty$\\ $\bullet \  \sum_{j=0}^{p}(b_{j}-b_{j+1})+ b_{p+1}=(b_{0}-b_{1})+ \sum_{j=1}^{p-1}(b_{j}-b_{j+1})+b_{p}=1 $ \\ The truncation error in \eqref{38e} is bounded, i.e.\\
\begin{equation}\label{3.55e}
  |l^{p+1}_{\Delta t}|\leq c \Delta t^{2-\gamma}
\end{equation} \\
where the constant $c$ is dependant on $y$. To apply this scheme, we need the values $y^{0}$ and $y^{1}$.\\
\\For $p=0$, \eqref{3.4e} takes the following form \\ \begin{equation}\label{3.6e}
                                          y^{1}(x)+\beta \alpha y^{1}_{xxxx}=v^{0}(x)+\beta u^{1}(x)
                                        \end{equation}\\
\\For $p=1$, \eqref{3.4e} becomes\\ \begin{equation*}\label{3.7e}
                                          y^{2}(x)+\beta \alpha y^{p+1}_{xxxx}=(b_{0}-b_{1})y^{1}(x)+b_{1}y^{0}(x)+\beta u^{2}(x)
                                        \end{equation*}\\
Now \eqref{3.4e} and \eqref{3.6e} with initial and boundary conditions formulate a complete set of semi-discrete problem for \eqref{1e}.\\
The error term $l^{p+1}$ can also be defined as \cite{lin2007}\\
\begin{equation}\label{3.8e}
 l^{p+1}=\beta \big(\frac{\partial^{\gamma}}{\partial t^{\gamma}}y(x,t_{p+1})-G^{\gamma}_{t}y(x,t_{p+1})\big).
 \end{equation}\\ From Eqs.\eqref{38e} and \eqref{3.55e}, the error term can be expressed as
                                        \begin{equation}\label{3.9e}
                                          |l^{p+1}|=\Gamma(2-\gamma)\Delta t^{\gamma}|l^{p+1}_{\Delta t}| \leq c_{y}\Delta t^2
                                        \end{equation}\\
 Now, we define some functional spaces and their standard norms as
                                        \begin{align*}
                                        H^{2}(\eta) &= \{g\in L^{2}(\eta), g_{x}, g_{xx} \in L^{2}(\eta)\}  \\
                                        H_{0}^{2}(\eta) &= \{g \in H^{2}(\eta), g|_{\partial \eta}=0, g_{x}|_{\partial \eta}=0\} \\
                                        H^{n}(\eta) &= \{g \in L^{2}(\eta), g^{(r)}_{x}, \forall r \leq n\}                                        \end{align*}
where $L^{2}(\eta)$ denotes the space of all measurable functions whose square is Lebesgue integrable in $\eta$ . The inner product and norm in $L^{2}(\eta)$ are given by
\begin{equation*}\label{3.10e}
<f,g>=\int_{\eta}fg dx,\ \ \ \|g\|_{0}=<g,g>^\frac{1}{2}
\end{equation*}\\
The inner product and norm in $S^{2}(\eta)$ are given by
\begin{equation*}\label{3.11e}
<f,g>_{2}=<f,g>+<f_{x},g_{x}>+<f_{xx},g_{xx}>,\ \ \ \|g\|_{2}=<g,g>_{2}^\frac{1}{2}
\end{equation*}\\
Also, the norm $\|.\|$ in $H^{n}(\eta)$ is defined in the following way
\begin{equation*}\label{3.12e}
 \|g\|_{n}= \big(\sum_{r=0}^{n}\|g_{x}^{(r)}\|_{0}^{2}\big)^{\frac{1}{2}}
 \end{equation*}
It is also preferred to define $\|.\|_{2}$ \\
\begin{equation}\label{3.13e}
\|g\|_{2}=\big(\|g\|_{0}^{2}+\beta \alpha \|g_{x}^{(2)}\|_{0}^{2}\big)^\frac{1}{2}
\end{equation}
Now, for the stability and convergence analysis, we are to find $y^{p+1}\in H_{0}^{2}(\eta)$ such that for all $g\in H_{0}^{2}(\eta)$,  Eqs.\eqref{3.4e} and \eqref{3.6e} give the following two relations
\begin{multline}\label{3.14e}
  <y^{p+1},g> +\beta \alpha <y_{xxxx}^{p+1},g>=(1-b_{1})<y^{p},g>+\sum_{j=1}^{p-1}(b_{j}-b_{j+1})<y^{p-j},g>\\+ b_{p}<y^{0},g>+\beta<u^{p+1},g>,
\end{multline}
and
  \begin{equation}\label{3.15e}
                              <y^{1},g>+\beta \alpha<y_{xxxx}^{1},g>=<y^{0},g> +\beta<u^{1},g>
                            \end{equation}
The theorem given below describes the unconditional stability of the semi--discrete problem.\\
\begin{theorem}\label{th1}
The discrete problem is unconditionally stable in such a way that $\forall \Delta t >0$, it holds
\ \ \ \ \ \ \ \ \ \ \ \ \ \ \begin{equation}\label{3.16e}
                              \|y^{p+1}\|_{2} \leq (\|y^{0}\|_{0}+\beta \sum_{j=1}^{p+1}\|u^{j}\|_{0}),\ \ p=0,1,2,\cdots,K-1
                            \end{equation}
where $\|.\|_{2}$ is discussed in Eq.\eqref{3.13e}.\\
\end{theorem}
\begin{proof}
  In order to prove this result, mathematical induction is used. For $p=0$ and $g=y^{1}, $ Eq. \eqref{3.15e} takes the following form
 \begin{equation*}\label{3.17e}
 <y^{1},y^{1}>+\beta \alpha <y^{1}_{xxxx},y^{1}> = <y^{0},y^{1}> +\beta <u^{1},y^{1}>
 \end{equation*}
Integrating by parts, the above result can be written as
\ \ \ \ \ \ \ \ \ \ \ \ \ \ \ \begin{equation}\label{3.18e}
<y^{1},y^{1}>+\beta \alpha <y^{1}_{xx},y^{1}_{xx}> = <y^{0},y^{1}> +\beta <u^{1},y^{1}>
                              \end{equation}
Due to the boundary conditions on $g$, all the boundary related contributions are disappeared.
From Schwarz inequality and the inequality $\|g\|_{0}$ $\leq$ $\|g\|_{2}$, Eq. \eqref{3.18e} becomes
                           \begin{align*}
                             \|y^{1}\|^{2}_{2} & \leq \|y^{0}\|_{0} \  \|y^{1}\|_{0} +\beta\|u^{1}\|_{0} \ \|y^{1}\|_{0} \\
                               & \leq \|y^{0}\|_{0} \  \|y^{1}\|_{2} +\beta\|u^{1}\|_{0} \ \|y^{1}\|_{2} \\
                             \|y^{1}\|_{2} & \leq (\|y^{0}\|_{0}+\beta\|u^{1}\|_{0})
                           \end{align*}
Suppose that the result is true for $g=y^{j}$ i.e
                            \begin{equation}\label{3.19e}
                              \|y^{j}\|_{2}\leq \big(\|y^{0}\|_{0}+\beta \sum_{i=1}^{j} \|u^{i}\|_{0}\big),\ \ \ \ j=2,3,\cdots,p.
                            \end{equation}
Let $g=y^{p+1}$ in Eq.\eqref{3.14e}
\begin{multline}
<y^{p+1},y^{p+1}>+\beta\alpha<y_{xxxx}^{p+1},y^{p+1}>=(1-b_{1})<y^{p},y^{p+1}>+\sum_{j=1}^{p-1}(b_{j}-b_{j+1})<y^{p-j},y^{p+1}>\\
+b_{p}<y^{0},y^{p+1}>+\beta <u^{p+1},y^{p+1}>
\end{multline}\\
Integrating by parts, we get
\begin{multline}
<y^{p+1},y^{p+1}>+\beta \alpha<y_{xx}^{p+1},y_{xx}^{p+1}>=(1-b_{1})<y^{p},y^{p+1}>+\sum_{j=1}^{p-1}(b_{j}-b_{j+1})<y^{p-j},y^{p+1}>\\
+b_{p}<y^{0},y^{p+1}>+\beta <u^{p+1},y^{p+1}>
\end{multline}\label{51e}
Again due to the boundary conditions on $g$ all the boundary related contributions are disappeared.
From Schwarz inequality and the inequality $\|g\|_{0}$ $\leq$ $\|g\|_{2}$, the above expression changes to
\begin{multline*}
  \|y^{p+1}\|_{2}^{2}\leq(1-b_{1})\|y^{p}\|_{0}\|y^{p+1}\|_{0}+\sum_{j=1}^{p-1}(b_{j}-b_{j+1}\|y^{p-j}\|_{0}\|y^{p+1}\|_{0}) \\
  +b_{p}\|y^{0}\|_{0}\|y^{p+1}\|_{0}+\beta \|u^{p+1}\|_{0}\|y^{p+1}\|_{0},
\end{multline*}
or
\begin{multline*}
   \|y^{p+1}\|_{2}^{2}\leq(1-b_{1})\|y^{p}\|_{0}\|y^{p+1}\|_{2}+\sum_{j=1}^{p-1}(b_{j}-b_{j+1}\|y^{p-j}\|_{0}\|y^{p+1}\|_{2}) \\
  +b_{p}\|y^{0}\|_{0}\|y^{p+1}\|_{2}+\beta \|u^{p+1}\|_{0}\|y^{p+1}\|_{2},
\end{multline*}
or
\begin{equation*}
   \|y^{p+1}\|_{2}\leq (1-b_{1})\|y^{p}\|_{0}+\sum_{j=1}^{p-1}(b_{j}-b_{j+1}\|y^{p-j}\|_{0} +b_{p}\|y^{0}\|_{0}+\beta \|u^{p+1}\|_{0}
\end{equation*}\\
Using \eqref{3.19e}, the above relation takes the following form\\
\begin{equation*}
  \|y^{p+1}\|_{2}\leq \bigg(|y^{0}\|_{0}+\beta \sum_{j=1}^{p-1}\|u^{j}\|_{0}\bigg)\bigg((1-b_{1})+\sum_{j=1}^{p-1}(b_{j}-b_{j+1})+b_{p}\bigg)+\beta \|u^{p+1}\|_{0}
\end{equation*}\\
Using the properties of $b_{j}$, we can write\\
\begin{equation*}
  \|y^{p+1}\|_{2} \leq \bigg(\|y^{0}\|_{0}+\beta \sum_{j=1}^{p+1}\|u^{j}\|_{0}\bigg)
\end{equation*}\\
\end{proof}
\begin{lemma}\label{L1}
 Let $\{y^{p}\}_{p=0}^{K}$ be the time discrete solution to Eqs. \eqref{3.14e}--\eqref{3.15e} and $y$ be the exact solution of \eqref{1e}, then
\begin{equation}\label{3.20e}
  \|y(t_{p})-y^{p}\|_{2}\leq c_{y}b_{p-1}^{-1}\Delta t^{2},\ \ \ \ p=1,2,\cdots,K.
  \end{equation}
\end{lemma}
\begin{proof}
Consider $e^{p}=y(x,t_{p})-y^{p}(x), $ for $p=1$, the error equation takes the following form by combining Eqs.\eqref{1e},\eqref{3.15e} and \eqref{3.13e} \
\begin{equation*}
  <e^{1},g>+\beta \alpha <e_{xx}^{1},g_{xx}>=<e^{0},g>+<l^{1},g>, \ \ \ \ \ \ \forall g\in S_{0}^{2}(\eta).
\end{equation*}\\
Let $g=e^{1}$ and $e^{0}=0$ gives the following relation\\
\begin{equation}\label{3.211e}
  \|e^{1}\|_{2} \leq \|l^{1}\|_{0}
\end{equation}\\
Eq. \eqref{3.9e} along with \eqref{3.211e}, gives
\begin{equation}\label{3.21e}
  \|y(t_{1})-y^{1}\|_{2} \leq c_{y} b_{0}^{-1}\Delta t^{2}.
\end{equation}\\
For $p=1$, Eq. \eqref{3.20e} is satisfied.\\
Next, suppose that \eqref{3.20e} is true for $p=1,2,3,\cdots,r.$ \ i.e.
 \begin{equation}\label{3.22e}
   \|y(t_{p})-y^{p}\|_{2}\leq c_{y}b_{p-1}^{-1}\Delta t^{2}
 \end{equation}\\
 Using \eqref{1e}, \eqref{3.13e} , \eqref{3.14e} and for $p=r+1$, the error equation is obtained as,\\
 \begin{multline}\label{3.212e}
   <e^{p+1},g>+\beta \alpha <e_{xx}^{p+1},g_{xx}>=(1-b_{1})<e^{p},g>+\sum_{j=1}^{p-1}(b_{j}-b_{j+1})<e^{p-j},g> \\
   +b_{p}<e^{0},g>+<l^{p+1},g>.
\end{multline}
Now, using the induction assumption and taking $g=e^{p+1}$ along with the relation  $\frac{b_{j}^{-1}}{b_{j+1}}<1$ for all positive integer $j$, Eq. \eqref{3.212e} can be written as\\
\begin{equation*}
  \|e^{p+1}\|_{2}\leq c_{y}b_{p}^{-1}\Delta t^{2}.
\end{equation*}\\
Hence, proved.
\end{proof}
Also, from the definition of $ b_{p}$, the following useful equation can be formulated  \\
\begin{align*}
  \underset{p\rightarrow\infty}{lim}\frac{b_{p-1}^{-1}}{p^{\gamma}} \leq & =\underset{p\rightarrow\infty}{lim}\frac{p^{-\gamma}}{p^{1-\gamma}-(p-1)^{1-\gamma}} \\
   & =\underset{p\rightarrow\infty}{lim}\frac{p^{-1}}{1-(1-\frac{1}{p})^{1-\gamma}} \\
   & =\frac{1}{1-\gamma}
\end{align*}\\
The function $\psi(z)$ is defined as $\psi(z)=\frac{z^{-\gamma}}{z^{1-\gamma}-(z-1)^{1-\gamma}}$,  as $\psi(z)\geq 0$ \ $\forall$  $z>1$, the function $\psi(z)$ is increasing on z. This indicates that as $1<p\rightarrow\infty$, $\frac{b_{p-1}^{-1}}{p^\gamma}$ increasingly approaches to $\frac{1}{1-\gamma}$. \\ Since, for $p=1$,  $p^{-\gamma}b_{p-1}^{-1}=1$. Therefore, it can be written in the following form\\
\begin{equation*}
  p^{-\gamma}b_{p-1}^{-1}\leq \frac{1}{1-\gamma} , \ \ \ \ \ \  p=1,2,\cdots,K.
\end{equation*}
Therefore, $\forall$ $p$ such that $p\Delta t \leq T$,\\
\begin{align*}
  \|y(t_{p})-y^{p}\|_{2} & \leq c_{y}b_{p-1}^{-1}\Delta t^{2} \\
   & =c_{y} p^{-\gamma}b_{p-1}^{-1}p^{-\gamma}\Delta t^{2-\gamma+\gamma} \\
   & \leq c_{y}\frac{1}{1-\gamma}(p\Delta t)^{\gamma}(\Delta t)^{2-\gamma} \\
   & \leq c_{y.\gamma}T^{\gamma}\Delta t^{2-\gamma}
\end{align*}
The above discussion can be summed up in following theorem.
\begin{theorem}\label{th2}
Let $y$ be the analytical exact solution to \eqref{1e} and $\{y^{p}\}_{p=0}^{K}$ be the time discrete solution to Eq.\eqref{3.14e} and Eq.\eqref{3.15e} subject to the initial condition $y^{0}=v_{0}(x)$, $x\in[0,L]$, then the following holds \\
\begin{equation}\label{3.23e}
  \|y(t_{p})-y^{p}\|_{2}\leq c_{y.\gamma}T^{\gamma}\Delta t^{2-\gamma} ,\ \ \ p=1,2,3,\cdots,K.
\end{equation}
\end{theorem}
 \section{Discretization in Space}
Let $(x_{i},t_{p})$ be the grid points which uniformly discretize the region $[0,L]\times[0,T]$ with $x_{i}=ih$, $t_{p}= p\Delta t$, $T=K\Delta t$, where, $i=0(1)n$ and $p=0(1)K$. The parameters $h$, $\Delta t$ are the grid sizes in the space and time directions respectively. The space discretization of Eq.\eqref{3.4e} using non polynomial quintic spline is formulated as\\
\begin{equation}\label{4.1e}
  S_{i}^{p+1}+\beta\alpha F^{p+1}=(1-b_{1})S_{i}^{p}+\sum_{j=1}^{p-1}(b_{j}-b_{j+1})S_{i}^{p-j}+b_{p}v_{i}+\beta u_{i}^{p+1}.
\end{equation}
The operator $\Phi$ is defined as
\begin{equation}\label{4.2e}
  \Phi S_{j}=\alpha_{1}S_{j-2}+\beta_{1}S_{j-1}+\gamma_{1}S_{j}+\beta_{1}S_{j+1}+\alpha_{1}S_{j+2}.
\end{equation}
Now, Eq.\eqref{26e} takes the following form\\
\begin{equation}\label{4.3e}
  \Phi F_{i}=\frac{1}{h^{4}}(S_{i-2}-4S_{i-1}+6S_{i}-4S_{i+1}+S_{i+2}).
\end{equation}
Applying the operator $\Phi$ on Eq.\eqref{4.1e}, we get the following result
\begin{multline}\label{4.4e}
  \alpha_{1}S_{i-2}^{p+1}+\beta_{1}S_{i-1}^{p+1}+\gamma_{1}S_{i}^{p+1}+\beta_{1}S_{i+1}^{p+1}+\alpha_{1}S_{i+2}^{p+1}+\frac{\beta\alpha}{h^{4}}(S_{i-2}^{p+1}-4S_{i-1}^{p+1}+6S_{i}^{p+1}-4S_{i+1}^{p+1}+S_{i+2}^{p+1}) \\ =(1-b_{1})(\alpha_{1}S_{i-2}^{p}+\beta_{1}S_{i-1}^{p}+\gamma_{1}S_{i}^{p}+\beta_{1}S_{i+1}^{p}+\alpha_{1}S_{i+2}^{p})+\sum_{j=1}^{p-1}(b_{j}-b_{j+1})(\alpha_{1}S_{i-2}^{p-j}+\beta_{1}S_{i-1}^{p-j}\\+\gamma_{1}S_{i}^{p-j}+\beta_{1}S_{i}^{p-j}
  +\alpha_{1}S_{i+2}^{p-j})+b_{p}(\alpha_{1}v_{i-2}+\beta_{1}v_{i-1}+\gamma_{1}v_{i}+\beta_{1}v_{i+1}+\alpha_{1}v_{i+2})\\+\beta(\alpha_{1}u_{i-2}^{p+1}+\beta_{1}u_{i-1}^{p+1}+\gamma_{1}u_{i}^{p+1}+\beta_{1}u_{i+1}^{p+1}+\alpha_{1}u_{i+2}^{p+1})
 ~~~, p=1,2,3,\cdots,K-1.
\end{multline}
After simplifying the system \eqref{4.4e} takes the following form
\begin{multline}\label{4.5e}
(\alpha_{1}+\frac{\beta\alpha}{h^{4}})S_{i-2}^{p+1}+(\beta_{1}-4\frac{\beta\alpha}{h^{4}})S_{i-1}^{p+1}+(\gamma_{1}+6\frac{\beta\alpha}{h^{4}})S_{i}^{p+1} +(\beta_{1}-4\frac{\beta\alpha}{h^{4}})S_{i+1}^{p+1}+(\alpha_{1}+\frac{\beta\alpha}{h^{4}})S_{i+2}^{p+1}\\
  =Q_{i},\ \ \ i=2,3,\cdots,n-2,\ \ \ \ p=1,2,\cdots,K-1.
\end{multline}
where
\begin{multline}\label{59e}
  Q_{i}= (1-b_{1})(\alpha_{1}S_{i-2}^{p}+\beta_{1}S_{i-1}^{p}+\gamma_{1}S_{i}^{p}+\beta_{1}S_{i+1}^{p}+\alpha_{1}S_{i+2}^{p})
  +\sum_{j=1}^{p-1}(b_{j}-b_{j+1})(\alpha_{1}S_{i-2}^{p-j}+\beta_{1}S_{i-1}^{p-j}\\
  +\gamma_{1}S_{i}^{p-j}+\beta_{1}S_{i}^{p-j}+\alpha_{1}S_{i+2}^{p-j})+b_{p}(\alpha_{1}v_{i-2}+\beta_{1}v_{i-1}+\gamma_{1}v_{i}+\beta_{1}v_{i+1}+\alpha_{1}v_{i+2})
  \\+\beta(\alpha_{1}u_{i-2}^{p+1}+\beta_{1}u_{i-1}^{p+1}+\gamma_{1}u_{i}^{p+1}+\beta_{1}u_{i+1}^{p+1}+\alpha_{1}u_{i+2}^{p+1})
\end{multline}\\
System \eqref{59e} provides $(n-3)$ equations involving $S_{i}^{p+1}, i=1,2,\cdots,n-1$. Therefore, we further need two equations for complete solution of $S_{i}^{p+1}$. The required two end conditions can be derived using simply supported boundary conditions as
\begin{multline}
 (\omega_{0}-2\frac{\beta\alpha}{h^{4}})S_{0}^{p+1}+(\omega_{1}+5\frac{\beta\alpha}{h^{4}})S_{1}^{p+1}+(\omega_{2}-4\frac{\beta\alpha}{h^{4}})S_{2}^{p+1}
  +(\omega_{3}+\frac{\beta\alpha}{h^{4}})S_{3}^{p+1}=(1-b_{1})(\omega_{0}S_{0}^{p}+\omega_{1}S_{1}^{p}+\omega_{2}S_{2}^{p}\\+\omega_{3}S_{3}^{p})
  +\sum_{j=1}^{p-1}(b_{j}-b_{j+1})(\omega_{0}S_{0}^{p-j}+\omega_{1}S_{1}^{p-j}+\omega_{2}S_{2}^{p-j}+\omega_{3}S_{3}^{p-j})
  +b_{p}(\omega_{0}v_{0}+\omega_{1}v_{1}+\omega_{2}v_{2}\\+\omega_{3}v_{3})
  +\beta(\omega_{0}u_{0}^{p+1}+\omega_{1}u_{1}^{p+1}
+\omega_{2}u_{2}^{p+1}+\omega_{3}u_{3}^{p+1})
\end{multline}
Similarly
\begin{multline}
  (\omega_{3}+\frac{\beta\alpha}{h^{4}})S_{n-3}^{p+1}+(\omega_{2}-4\frac{\beta\alpha}{h^{4}})S_{n-2}^{p+1}+(\omega_{1}+5\frac{\beta\alpha}{h^{4}})S_{n-1}^{p+1}
  +(\omega_{0}-2\frac{\beta\alpha}{h^{4}})S_{n}^{p+1}=(1-b_{1})(\omega_{3}S_{n-3}^{p}+\omega_{2}S_{n-2}^{p}\\+\omega_{1}S_{n-1}^{p}+\omega_{0}S_{n}^{p})
  +\sum_{j=1}^{p-1}(b_{j}-b_{j+1})(\omega_{3}S_{n-3}^{p-j}+\omega_{2}S_{n-2}^{p-j}+\omega_{1}S_{n-1}^{p-j}+\omega_{0}S_{n}^{p-j})
  +b_{p}(\omega_{3}v_{n-3}+\omega_{2}v_{n-2}\\+\omega_{1}v_{n-1}+\omega_{0}v_{n})
  +\beta(\omega_{3}u_{n-3}^{p+1}+\omega_{2}u_{n-2}^{p+1}
  +\omega_{1}u_{n-2}^{p+1}+\omega_{0}u_{n-2}^{p+1})
\end{multline}\\
The proposed algorithm is a five point scheme. In order to implement it, the numerical values of   $S^{2}=[S_{1}^{2},S_{2}^{2},S_{3}^{2},\cdots,S_{n-1}^{2}]^{T}$ and $S^{1}=[S_{1}^{1},S_{2}^{1},S_{3}^{1},\cdots,S_{n-1}^{1}]^{T}$ are needed. To calculate the values of $S^{2}$, it is required to find $S^{1}$. Solving Eq.\eqref{3.6e} and using the non polynomial quintic spline technique, value of $ S^{1}$ can be found as:\\
\begin{multline}\label{4.8e}
  (\alpha_{1}+\frac{\beta\alpha}{h^{4}})S_{i-2}^{1}+(\beta_{1}-4\frac{\beta\alpha}{h^{4}})S_{i-1}^{1}+(\gamma_{1}+6\frac{\beta\alpha}{h^{4}})S_{i}^{1}+(\beta_{1}-4\frac{\beta\alpha}{h^{4}})S_{i+1}^{1}+(\alpha_{1}+\frac{\beta\alpha}{h^{4}})S_{i+2}^{1} \\
  =J_{i},\ \ \ \ \ \ i=2,3,\cdots,n-2.
\end{multline}
where
\begin{multline*}
  J_{i}=(\alpha_{1}v_{i-2}+\beta_{1}v_{i-1}+\gamma_{1}v_{i}+\beta_{1}v_{i+1}+\alpha_{1}v_{i+2})+\beta(\alpha_{1}u_{i-2}^{p+1}+\beta_{1}u_{i-1}^{1}+
\gamma_{1}u_{i}^{1}+\beta_{1}u_{i+1}^{1}+\alpha_{1}u_{i+2}^{1})
\end{multline*}
The system \eqref{4.8e} consists of $(n-3)$ equations involving $S_{i}^{1},i=1,2,\cdots,n-1$. Hence, to get a unique solution to this system, two additional end equations can be obtained from simply supported boundary conditions in the following way
\begin{multline}\label{4.9e}
(\omega_{0}-2\frac{\beta\alpha}{h^{4}})S_{0}^{1}+(\omega_{1}+5\frac{\beta\alpha}{h^{4}})S_{1}^{1}+(\omega_{2}-4\frac{\beta\alpha}{h^{4}})S_{2}^{1}+(\omega_{3}+\frac{\beta\alpha}{h^{4}})S_{3}^{1}=(\omega_{0}v_{0}+\omega_{1}v_{1}+\omega_{2}v_{2}+\omega_{3}v_{3})\\
+\beta(\omega_{0}u_{0}^{1}+\omega_{1}u_{1}^{1}+\omega_{2}u_{2}^{1}+\omega_{3}u_{3}^{1})
\end{multline}
\begin{multline}\label{4.10e}
  (\omega_{3}+\frac{\beta\alpha}{h^{4}})S_{n-3}^{1}+(\omega_{2}-4\frac{\beta\alpha}{h^{4}})S_{n-2}^{1}+(\omega_{1}+5\frac{\beta\alpha}{h^{4}})S_{n-1}^{1}+(\omega_{0}-2\frac{\beta\alpha}{h^{4}})S_{n}^{1}
=(\omega_{3}v_{n-3}+\omega_{2}v_{n-2}\\+\omega_{1}v_{n-1}+\omega_{0}v_{n})+\beta(\omega_{3}u_{n-3}^{1}+\omega_{2}u_{n-2}^{1}+\omega_{1}u_{n-1}^{1}+\omega_{0}u_{n}^{1})
\end{multline}
Suppose $v=[v_{1},v_{2},\cdots,v_{n-1}]^{T}$, $\ u= [u_{1},u_{2},\cdots,u_{n-1}]^{T}$, $\overset{\sim}{v}=[v_{0},0,\cdots,0,v_{n}]^{T}$ and $\overset{\sim}{u}=[u_{0},0,\cdots,0,u_{n}]^{T}$ are column vectors with dimension $(n-1)$.
The system in \eqref{4.8e}--\eqref{4.10e} can be expressed as
\begin{equation*}
  AS^{1}=B(v+\beta u)+C(\overset{\sim}{v}+\beta \overset{\sim}{u})
\end{equation*}
where A,B and C are square matrices of order $(n-1)$, such that
\begin{equation*}
  A=\begin{pmatrix}
      \omega_{1}+5\frac{\beta\alpha}{h^{4}} & \omega_{2}-4\frac{\beta\alpha}{h^{4}} & \omega_{3}+\frac{\beta\alpha}{h^{4}} & 0 & 0 & 0 & \cdots & 0\\
    \beta_{1}-4\frac{\beta\alpha}{h^{4}} & \gamma_{1}+6\frac{\beta\alpha}{h^{4}} & \beta_{1}-4\frac{\beta\alpha}{h^{4}} & \alpha_{1}+\frac{\beta\alpha}{h^{4}} & 0 & 0 & \cdots & 0\\
    \alpha_{1}+\frac{\beta\alpha}{h^{4}} & \beta_{1}

    -4\frac{\beta\alpha}{h^{4}} & \gamma_{1}+6\frac{\beta\alpha}{h^{4}} & \beta_{1}-4\frac{\beta\alpha}{h^{4}} & \alpha_{1}+\frac{\beta\alpha}{h^{4}} & 0 & \cdots & 0 \\
     & \ddots &  &  & \ddots & & \ddots   \\
    0 & \cdots & 0 & \alpha_{1}+\frac{\beta\alpha}{h^{4}} & \beta_{1}-4\frac{\beta\alpha}{h^{4}} & \gamma_{1}+6\frac{\beta\alpha}{h^{4}} & 0 \beta_{1}-4\frac{\beta\alpha}{h^{4}} & \alpha_{1}+\frac{\beta\alpha}{h^{4}} \\
    0 & \cdots & 0 & & \alpha_{1}+\frac{\beta\alpha}{h^{4}} & \beta_{1}-4\frac{\beta\alpha}{h^{4}} & \gamma_{1}+6\frac{\beta\alpha}{h^{4}} & \beta_{1}-4\frac{\beta\alpha}{h^{4}} \\
    0 & \cdots & 0 & 0 & 0 & \omega_{3}+\frac{\beta\alpha}{h^{4}} & \omega_{2}-4\frac{\beta\alpha}{h^{4}} & \omega_{1}+5\frac{\beta\alpha}{h^{4}} \\
    \end{pmatrix}
\end{equation*}\\
\begin{equation*}
  B=\begin{pmatrix}
      \omega_{1} & \omega_{2} & \omega_{3} & 0 & 0 & 0 & \cdots & 0   \\
    \alpha_{1} & \beta_{1} & \gamma_{1} & \alpha_{1} & 0 & 0 & \cdots & 0\\
    \alpha_{1} & \beta_{1} & \gamma_{1} & \beta_{1} & \alpha_{1} & 0 & \cdots & 0 \\
      & \ddots &  &  & \ddots & & \ddots  \\
   0 & \cdots & 0 & \alpha_{1} & \beta_{1} & \gamma_{1} & \beta_{1} & \alpha_{1}  \\
   0 & \cdots & 0 & 0 & \alpha_{1} & \beta_{1} & \gamma_{1} & \beta_{1} \\
   0 & \cdots & 0 & 0 & 0 & \omega_{3} & \omega_{2} & \omega_{1} \\
    \end{pmatrix}
    \ \text{and} \
    C=\begin{pmatrix}
  \omega_{0} & 0 & 0 &  0 & 0 & 0 & 0 & \cdots & 0   \\
    1 & 0 & 0 & 0 & 0 & 0 & 0 & \cdots & 0 \\
    0 & 0 & 0 & 0 & 0 & 0 & 0 & \cdots & 0 \\
    & \ddots &  &  & \ddots & & \ddots  \\

   0 & \cdots & 0 & 0 & 0 & 0 & 0 & 0 & 0 \\
     0 & \cdots & 0 & 0 & 0 & 0 & 0 & 0 & 1 \\
    0 & \cdots & 0 & 0 & 0 & 0 & 0 & 0 & \omega_{0} \\
    \end{pmatrix}
\end{equation*}
\subsection{Calculation of Truncation Error}
The Eq.\eqref{4.4e} can be written in the following form
\begin{multline}\label{61e}
 h^{4}(\alpha_{1}S_{i-2}^{p+1}+\beta_{1}S_{i-1}^{p+1}+\gamma_{1}S_{i}^{p+1}+\beta_{1}S_{i+1}^{p+1}+\alpha_{1}S_{i+2}^{p+1})
 +\beta\alpha(S_{i-2}^{p+1}-4S_{i-1}^{p+1}+6S_{i}^{p+1}-4S_{i+1}^{p+1}\\
+S_{i+2}^{p+1})=h^{4}(1-b_{1})(\alpha_{1}S_{i-2}^{p}+\beta_{1}S_{i-1}^{p}+\gamma_{1}S_{i}^{p}+\beta_{1}S_{i+1}^{p}+\alpha_{1}S_{i+2}^{p})
+\sum_{j=1}^{p-1}h^{4}(b_{j}-b_{j+1})\\
+(\alpha_{1}S_{i-2}^{p-j}+\beta_{1}S_{i-1}^{p-j}+\gamma_{1}S_{i}^{p-j}+\beta_{1}S_{i}^{p-j}+\alpha_{1}S_{i+2}^{p-j})
  +h^{4}b_{p}(\alpha_{1}v_{i-2}+\beta_{1}v_{i-1}+\gamma_{1}v_{i}\\
  +\beta_{1}v_{i+1}+\alpha_{1}v_{i+2})+h^{4} \beta(\alpha_{1}u_{i-2}^{p+1}+\beta_{1}u_{i-1}^{p+1}+\gamma_{1}u_{i}^{p+1}
 +\beta_{1}u_{i+1}^{p+1}+\alpha_{1}u_{i+2}^{p+1})\\p=1,2,3,\cdots,K-1
\end{multline}
or
\begin{multline*}
h^{4}(\alpha_{1}S_{i-2}^{p+1}+\beta_{1}S_{i-1}^{p+1}+\gamma_{1}S_{i}^{p+1}+\beta_{1}S_{i+1}^{p+1}+\alpha_{1}S_{i+2}^{p+1})
 +\beta\alpha(S_{i-2}^{p+1}-4S_{i-1}^{p+1}+6S_{i}^{p+1}-4S_{i+1}^{p+1}\\
 +S_{i+2}^{p+1})=h^{4}(1-b_{1})(\alpha_{1}S_{i-2}^{p}+\beta_{1}S_{i-1}^{p}+\gamma_{1}S_{i}^{p}+\beta_{1}S_{i+1}^{p}+\alpha_{1}S_{i+2}^{p})
 +\sum_{j=1}^{p-1}h^{4}(b_{j}-b_{j+1})(\alpha_{1}S_{i-2}^{p-j}\\
 +\beta_{1}S_{i-1}^{p-j}+\gamma_{1}S_{i}^{p-j}+\beta_{1}S_{i}^{p-j}+\alpha_{1}S_{i+2}^{p-j})+h^{4}b_{p}(\alpha_{1}v_{i-2}+\beta_{1}v_{i-1}
+\gamma_{1}v_{i}+\beta_{1}v_{i+1}+\alpha_{1}v_{i+2})\\
+h^{4}\beta(\alpha_{1}u_{i-2}^{p+1}+\beta_{1}u_{i-1}^{p+1}+\gamma_{1}u_{i}^{p+1}+\beta_{1}u_{i+1}^{p+1}+\alpha_{1}u_{i+2}^{p+1}),
~~~p=1,2,3,\cdots,K-1
\end{multline*}
Expanding equation \eqref{4.4e} with Taylor series in terms of $S(x_{i},t_{p})$ and its spatial derivatives , the truncation error is obtained,as
\begin{multline}
  T_{i}=\big(\mu_{1}h^{4}+\mu_{2}h^{6}D_{x}^{2}+\mu_{3}h^{8}D_{x}^{4}+\cdots \big)S_{i}^{p+1}+ \beta\alpha\big(h^{4}D_{x}^{4}+\frac{1}{6}h^{6}D_{x}^{6}+\frac{1}{80}h^{8}D_{x}^{8}+\cdots\big)S_{i}^{p+1}\\
  - \big((1-b_{1})\mu_{1}h^{4}+\mu_{2}h^{6}D_{x}^{2}+\mu_{3}h^{8}D_{x}^{4}+\cdots\big)S_{i}^{p}
  - \big(\sum_{j=1}^{p-1}(b_{j}-b_{j+1})\mu_{1}h^{4}+\mu_{2}h^{6}D_{x}^{2}+\mu_{3}h^{8}D_{x}^{4}+\cdots \big)S_{i}^{p-j}\\
  - b_{p}\big(\mu_{1}h^{4}+\mu_{2}h^{6}D_{x}^{2}+\mu_{3}h^{8}D_{x}^{4}+\cdots \big)v_{i}
  -\beta\big(\mu_{1}h^{4}+\mu_{2}h^{6}D_{x}^{2}+\mu_{3}h^{8}D_{x}^{4}+\cdots \big)u_{i}^{p+1}\\
  \end{multline}
  Where\\
  $\mu_{1}=2\alpha_{1}+2\beta_{1}+\gamma_{1}$,
   $\mu_{2}=4\alpha_{1}+\beta_{1}$ and
  $\mu_{3}=\frac{4}{3}\alpha_{1}+\beta_{1}.$
 \begin{multline*}
  T_{i} =\big(\mu_{1}h^{4}+\mu_{2}h^{6}D_{x}^{2}+\mu_{3}h^{8}D_{x}^{4}+\cdots \big)S_{i}^{p+1}+\beta\alpha\big(h^{4}D_{x}^{4}+\frac{1}{6}h^{6}D_{x}^{6}+\frac{1}{80}h^{8}D_{x}^{8}+\cdots\big)S_{i}^{p+1}\\
  - \big((1-b_{1})\mu_{1}h^{4}+\mu_{2}h^{6}D_{x}^{2}+\mu_{3}h^{8}D_{x}^{4}+\cdots\big)S_{i}^{p}
  - \bigg(\sum_{j=1}^{p-1}(b_{j}-b_{j+1})\mu_{1}h^{4}+\mu_{2}h^{6}D_{x}^{2}+\mu_{3}h^{8}D_{x}^{4}+\cdots \bigg)S_{i}^{p-j}\\
  - b_{p}\big(\mu_{1}h^{4}+\mu_{2}h^{6}D_{x}^{2}+\mu_{3}h^{8}D_{x}^{4}+\cdots \big)S_{i}^{0}
  -\beta\bigg(\mu_{1}h^{4}+\mu_{2}h^{6}D_{x}^{2}+\mu_{3}h^{8}D_{x}^{4}+\cdots \bigg)(D_{t}^{2-\gamma}+\alpha D_{x}^{4})S_{i}^{p+1}
  \end{multline*}
  From above discussion and Theorem 2, It is concluded that the scheme is of $O(h^{4}+\Delta t^{4-\gamma})$

\section{Numerical Results}In this section, we consider three test problems to check the validity and efficiency of the proposed numerical scheme. The approximate results are compared with quintic spline collocation method (QnSM) used in [19]. All the computations are executed in $Mathematica\  9.0$. The accuracy of presented technique is tested by error norms $L_{\infty}$, $L_{2}$ and order of convergence ($\chi$), which are calculated as
\begin{equation*}
  L_{\infty}= max|y_{i}-Y_{i}|~,\ \ \ \  \ \
  L_{2}=\sqrt{\frac{\sum_{i=0}^{n}|y_{i}-Y_{i}|^{2}}{\sum_{i=0}^{n}|y_{i}|^{2}}}
 \end{equation*}
 \begin{equation*}
    \chi  =\frac{1}{\log(2)}\bigg[\log\frac{L_{\infty}(n)}{L_{\infty}(2n)}\bigg]\\
 \end{equation*}
 where $y_{i},Y_{i}$ represent the exact and approximate solution at $i^{th}$ knot respectively.
\begin{problem}\label{prb1}
  Consider the fourth order time-fractional PDE \cite{tariq2017}
\begin{equation*}
\frac{\partial^{\gamma}y}{\partial t^{\gamma}}+\alpha \frac{\partial^{4}y}{\partial x^{4}}= u(x,t), \ \ \ \ \ 0\leq x\leq 1,\ \ \ \ 0<t\leq T
\end{equation*}
with initial condition\\
\begin{equation*}
  y(x,0)=\sin(\pi x)
\end{equation*}
and the boundary conditions\\
\begin{align*}
  y(0,t)&=y(1,t)=0 \\
  y_{xx}(0,t) & =y_{xx}(1,t)=0
\end{align*}
\end{problem}
The exact solution is  $y(x,t)=\sin(\pi x)e^{t}.$
The computational error norms $L_{\infty}$ and $L_{2}$  corresponding to different values of $\gamma$ are listed in Table \ref{t11} when $\alpha=0.01$ and $n=100$. It is obvious that our proposed computational approach produces more accurate results with $\Delta t=0.01$ as compared to QnSM used in \cite{tariq2017} with $\Delta t=0.000001$. A comparison of $L_{\infty}$, $L_{2}$ and order of convergence $\chi$ with QnSM \cite{tariq2017} at $t=1$ corresponding to $\gamma=0.5$ and $\Delta t=h$ is reported in Table \ref{t12}. It is observed that the order of convergence in numerical results exhibits a good agreement with the theoretical estimation. In Figure \ref{fig:f1}, three dimensional visuals of exact and approximate solutions are displayed for $n=100$, $\Delta t=0.01$. The absolute numerical error at $t=1$ corresponding to $n=100$, $\Delta t=0.01$ and $\gamma=0.5$ is portrayed in Figure \ref{fig:f1e}.
\begin{table}[h!]
  \centering
  \caption{Comparison of absolute error for Problem \ref{prb1} when $n=100$}\label{t11}
  \begin{tabular}{cccccc}
  \hline\hline
   & \multicolumn{2}{c}{Method in \cite{tariq2017}} && \multicolumn{2}{c}{Proposed method} \\
   & \multicolumn{2}{c}{$\Delta t=0.000001$, $t=0.0001$} && \multicolumn{2}{c}{$\Delta t=0.01$, $t=1$} \\
  \cline{2-3}\cline{5-6}
  $\gamma$& $L_{\infty}$ & $L_{2}$ && $L_{\infty}$ & $L_{2}$\\
  \hline
  0.25 & $1.2346\times10^{-5}$ & $8.7299\times10^{-7}$ && $6.4023\times10^{-8}$ & $1.6489\times10^{-8}$ \\
  [0.2cm]
  0.50& $1.7841\times10^{-6}$ & $1.2616\times10^{-7}$ && $5.6896\times10^{-8}$ & $1.7455\times10^{-8}$ \\
  [0.2cm]
  0.75& $5.1222\times10^{-7}$ & $6.6219\times10^{-8}$ && $ 4.0157\times10^{-8}$ & $1.3770\times10^{-9}$\\
  [0.2cm]
    1.00 &$ 9.2130\times10^{-7}$ &$6.5146\times10^{-8}$&& $9.0571\times10^{-9} $& $3.3451\times10^{-9}$\\
  \hline\hline
\end{tabular}
\end{table}

\begin{table}
  \centering
  \caption{Computational error norms and order of convergence for Problem \ref{prb1} when $\Delta t=h $}\label{t12}
  \resizebox{\textwidth}{!}{
  \begin{tabular}{cccccccccc}
  \hline\hline
  & \multicolumn{4}{c}{Method in \cite{tariq2017}} && \multicolumn{4}{c}{Proposed method} \\
   & \multicolumn{4}{c}{$\Delta t=0.000001$, $t=0.0001$} && \multicolumn{4}{c}{$\Delta t=0.01$, $t=1$} \\
  \cline{2-5}\cline{7-10}
   $n$ &  $L_{\infty}$          & $\chi$ & $L_{2}$ & $\chi$ && $L_{\infty}$  & $\chi$ & $L_{2}$ & $\chi$ \\
   \hline
 10& -- & -- & -- &$ -$ && $5.3290\times10^{-4}$ & -- & $1.5741\times10^{-4}$  &-- \\
  20&$1.3279\times10^{-2}$ & --& $2.0995\times10^{-3}$ & -- && $2.9936\times10^{-5}$ &$4.1539$ & $8.2199\times10^{-4}$  &$4.2592$ \\
  40&$4.4730\times10^{-3}$ & $1.5700$ & $8.0009\times10^{-4}$ & $1.3918$ && $1.5100\times10^{-6}$ & $4.3092$ & $4.5177\times10^{-8}$ & $4.1873$ \\
  80&$1.5282\times10^{-3}$ & $1.5496$ &$3.2081\times10^{-5}$ & $1.3183$  && $7.5135\times10^{-8}$ & $4.3289$ & $2.7099\times10^{-8}$ & $4.0573$ \\
160& $5.2715\times10^{-4}$ & $1.5357$ &$9.9469\times10^{-5}$ & $1.6892$   && $4.2651\times10^{-9}$ & $4.1388$ & $1.1.6778\times10^{-9}$ & $4.0136$ \\
  \hline\hline
\end{tabular}
}
\end{table}
\begin{figure}[h!]
\minipage{0.45\textwidth}
\centering
  (a) Exact solution
\endminipage\hfill
\minipage{0.45\textwidth}
\centering
  (b) Approximate solution
\endminipage\hfill
\caption{Exact and approximate solution for Problem \ref{prb1} when $n=100,~\Delta t=0.01$ and $\gamma=0.25$}\label{fig:f1}
\end{figure}
\begin{figure}[h!]
\centering
\caption{Absolute error for Problem \ref{prb1} using $n=100$,~$\Delta t=0.01$ and $\gamma=0.25$}\label{fig:f1e}
\end{figure}
\begin{problem}\label{prb2}
  Consider the following fourth order time-fractional PDE \cite{tariq2017}
\begin{equation*}
\frac{\partial^{\gamma}y}{\partial t^{\gamma}}+0.05 \frac{\partial^{4}y}{\partial x^{4}}= u(x,t), \ \ \ \ \ 0\leq x\leq 1,\ \ \ \ 0<t\leq T
\end{equation*}
with initial condition
\begin{equation*}
  y(x,0)=0
\end{equation*}
and the boundary conditions
\begin{align*}
  y(0,t)&=y(1,t)=0 \\
  y_{xx}(0,t)&= y_{xx}(1,t)=0
\end{align*}
\end{problem}
The close form solution is $y(x,t)=t\sin(\pi x).$
 The computational error norms $L_{\infty}$ and $L_{2}$ for $n=100$ and different choices of $\gamma$ are presented in Table \ref{t21}. It can be observed that our presented approach yields more accurate results with $\Delta t=0.01$ as compared to QnSM employed in \cite{tariq2017} with $\Delta t=0.000001$. Table \ref{t22} presents the error norms $L_{\infty}$, $L_{2}$ and the corresponding order of convergence at $t=1$ for $\Delta t=h and \gamma=0.5$. Figure \ref{fig:f2} shows three dimensional plots of exact and approximate solutions for $n=100$, $\Delta t=0.01$ and $\gamma=0.5$. The absolute numerical error at $t=1$ corresponding to $n=100$, $\Delta t=0.01$ and $\gamma=0.25$ is portrayed in Figure \ref{fig:f2e}.
\begin{table}
  \centering
  \caption{Comparison of absolute errors for problem \ref{prb2} when n=100}\label{t21}
  \begin{tabular}{cccccc}
  \hline\hline
  & \multicolumn{2}{c}{Method in \cite{tariq2017}} && \multicolumn{2}{c}{Proposed method} \\
   & \multicolumn{2}{c}{$\Delta t=0.000001$, $t=0.0001$} && \multicolumn{2}{c}{$\Delta t=0.01$, $t=1$} \\
  \cline{2-3}\cline{5-6}
  & $L_{\infty}$ & $L_{2}$ && $L_{\infty}$ & $L_{2}$\\
  \hline
  0.25 & $9.6400\times10^{-7}$ & $6.8165\times10^{-8}$ && $9.5143\times10^{-9}$ & $9.2399\times10^{-9}$ \\
  0.50& $9.9135\times10^{-7}$ & $7.0099\times10^{-8}$ && $8.9541\times10^{-9}$ & $8.4093\times10^{-9}$ \\
  0.75& $9.9997\times10^{-7}$ & $7.0709\times10^{-8}$ &&$ 9.7525\times10^{-9}$ & $8.2487\times10^{-9}$\\
    1.00 &$ 1.0000\times10^{-6}$ &$7.0711\times10^{-8} $&& $9.8911\times10^{-9} $& $9.0927\times10^{-9}$\\
  \hline\hline
\end{tabular}
\end{table}
\begin{table}
  \centering
  \caption{Computational error norms and order of convergence for problem \ref{prb2} when $\gamma=0.5$}\label{t22}
  \resizebox{\textwidth}{!}{
  \begin{tabular}{cccccccccc}
  \hline\hline
  & \multicolumn{4}{c}{Method in \cite{tariq2017}} && \multicolumn{4}{c}{Proposed method} \\
   & \multicolumn{4}{c}{$\Delta t=0.000001$, $t=0.0001$} && \multicolumn{4}{c}{$\Delta t=0.01$, $t=1$} \\
  \cline{2-5}\cline{7-10}
   $n$ &  $L_{\infty}$          & $\chi$ & $L_{2}$ & $\chi$ && $L_{\infty}$  & $\chi$ & $L_{2}$ & $\chi$ \\
 \hline
  $20$ & $1.0861\times10^{-2}$ & -- & $1.7173\times10^{-3}$ & -- && $9.5789\times10^{-4}$ & -- & $9.4891\times10^{-4}$ &  \\
  $40$ & $3.9014\times10^{-3}$ & $1.4771$ & $5.3619\times10^{-4}$ & $1.6793$ &&$4.9968\times10^{-5}$ & $4.2606$  & $4.8987\times10^{-5}$ & $4.2763$ \\
  $80$ & $1.3892\times10^{-3}$ & $1.4897$ & $2.0983\times10^{-4}$ & $1.3535$ && $2.4519\times10^{-6}$ & $4.3490$ & $2.4361\times10^{-6}$ & $4.3291$\\
  $160$& $4.9276\times10^{-4}$ & $1.4953$ & $6.7546\times10^{-5}$ & $1.6353$ && $1.2826\times10^{-7}$ & $4.2567$ & $1.2943\times10^{-7}$ & $4.2578$\\
  \hline\hline
\end{tabular}
}
\end{table}
\begin{figure}[h!]
\minipage{0.45\textwidth}
\centering
  (a) Exact solution
\endminipage\hfill
\minipage{0.45\textwidth}
\centering
  (b) Approximate solution
\endminipage\hfill
\caption{Exact and approximate solution for Problem \ref{prb2} when $n=100$,~$\Delta t=0.01$ and $\gamma=0.5$ }\label{fig:f2}
\end{figure}
\begin{figure}[h!]
\centering
\caption{Absolute error for Problem \ref{prb2} when $n=100$,~ $\Delta t=0.01$ and $\gamma=0.5$}\label{fig:f2e}
\end{figure}
\begin{problem}\label{prb3}
Consider the fourth order time--fractional PDE \cite{tariq2017}\\
\begin{equation*}
\frac{\partial^{\gamma}y}{\partial t^{\gamma}}+0.05 \frac{\partial^{4}y}{\partial x^{4}}= u(x,t), \ \ \ \ \ 0\leq x\leq 1,\ \ \ \ 0<t\leq T
\end{equation*}
with initial condition\\
\begin{equation*}
  y(x,0)=\sin\pi x,
\end{equation*}
and the boundary conditions\\
\begin{align*}
  y(0,t)&=y(1,t)=0 \\
  y_{xx}(0,t) & =y_{xx}(1,t)=0,
\end{align*}
\end{problem}
The analytical exact solution is $y(x,t)=(t+1)\sin\pi x$. Table \ref{t31} shows a comparison of computational error norms with QnSM \cite{tariq2017} corresponding to different selections of $\gamma$. It is found that our computational outcomes are better than QnSM \cite{tariq2017}. In Figure \ref{fig:f3}, three dimensional visuals of exact and approximate solutions are displayed for $n=40$, $\gamma=0.5$ and $\Delta t=0.01$. The absolute numerical error at $t=1$ corresponding to $n=40$, $\Delta t=0.01$ and $\gamma=0.5$ is portrayed in Figure \ref{fig:f3e}. It is obvious that approximate solution is highly consistent with the analytical exact solution, which proves the effectiveness of proposed scheme.
\begin{table}\label{t31}
  \centering
  \caption{Comparison of absolute errors for problem \ref{prb3} when $n=40$}
   \begin{tabular}{cccccc}
  \hline\hline
  & \multicolumn{2}{c}{Method in \cite{tariq2017}} && \multicolumn{2}{c}{Proposed method} \\
   & \multicolumn{2}{c}{$\Delta t=0.000001$, $t=0.0001$} && \multicolumn{2}{c}{$\Delta t=0.01$, $t=1$} \\
  \cline{2-3}\cline{5-6}
  & $L_{\infty}$ & $L_{2}$ && $L_{\infty}$ & $L_{2}$\\
  \hline
  0.25 & -- & -- && $8.7834\times10^{-8}$ & $7.1567\times10^{-8}$ \\
  0.50 & $2.1423\times10^{-4}$ & $2.3952\times10^{-5}$ && $3.5666\times10^{-8}$ & $1.5945\times10^{-8}$ \\
  0.75 & -- & -- &&$ 5.5943\times10^{-8}$ & $3.4214\times10^{-8}$\\
    1.00 &$ 2.6524\times10^{-5}$ &$2.9654\times10^{-6} $&& $1.4199\times10^{-8} $& $7.1879\times10^{-9}$\\
  \hline\hline
\end{tabular}
\end{table}
\begin{figure}[h!]
\minipage{0.45\textwidth}
\centering
  (a) Exact solution
\endminipage\hfill
\minipage{0.45\textwidth}
\centering
  (b) Approximate solution
\endminipage\hfill
\caption{Exact and approximate solution for Problem \ref{prb3} when $\gamma=0.5,~n=40$ and $\Delta t=0.01$ }\label{fig:f3}
\end{figure}
\begin{figure}[h!]
\centering
\caption{Absolute error for Problem \ref{prb3} when $n=40,~\Delta t=0.01$ and $\gamma=0.5$}\label{fig:f3e}
\end{figure}
\section{Conclusion}
In this work, non polynomial quintic spline collocation method has been employed for approximate solution of fourth order time--fractional partial differential equations. The backward Euler's method has been used for temporal discretization, whereas, non polynomial quintic spline function composed of a trigonometric part and a polynomial part has been employed to interpolate the unknown function in spatial direction. The proposed numerical algorithm is proved to be convergent and unconditionally stable. The numerical outcomes are found to be more accurate as compared to QnSM \cite{tariq2017}.

\bibliographystyle{unsrt}
\bibliography{bibfile_amin}
\end{document}